\documentclass[12pt,reqno]{article}
\usepackage{graphicx} 
\usepackage{amsmath}

\title{Normal Ordering in the Algebra Generated by $x$ and $\mathrm{I}$ and a Combinatorial Generalization of Bessel Numbers}

\author{Abdelhay Benmoussa\\
Anssis School Complex\\
Tiqqi, Agadir Ida Outanane\\
Morocco\\
\texttt{abdelhay.benmoussa@taalim.ma}}

\begin{document}

\maketitle  

\newtheorem{corollary}{Corollary}
\newtheorem{theorem}{Theorem}
\newtheorem{proof}{Proof}
\begin{abstract}
We investigate the algebra generated by the operators $x$ and $\mathrm{I} = \int_0^x$, which satisfy the commutation relation
\[
[\mathrm{I},x] = \mathrm{I}x - x\mathrm{I} = - \mathrm{I}^2.
\]  
We develop a combinatorial framework for the normal ordering of words in this algebra and show that any word can be written in the form
\[
w = \sum_{i,j} c(i,j) \, x^i \mathrm{I}^j,
\]
where the coefficients $c(i,j)$ are signed integers. Focusing on powers of the operator $(x\mathrm{I})^n$, we demonstrate that the corresponding coefficients coincide with the classical Bessel numbers (OEIS A001498). We further extend this analysis to powers of the generalized operators $(x^\lambda \mathrm{I}^\delta)^n$ and, finally, provide an explicit normal-ordered expression for an arbitrary word.
\end{abstract}

\section{Introduction}

Normal ordering problems in operator algebras have a long history, beginning with the classical Weyl algebra. In the Weyl algebra, the Scherk identity (1823) states
\[
(x D)^n = \sum_{k=0}^{n} S(n,k) x^k D^k,
\]
where $S(n,k)$ are the Stirling numbers of the second kind.  

In this paper, we consider an analogous algebra generated by $x$ and $\mathrm{I} = \int_0^x$ with commutation relation
\[
[\mathrm{I}, x] = -\mathrm{I}^2,
\]
and study the normal ordering of powers of $(x\mathrm{I})$ and its generalizations.

\section{Combinatorial Framework}

\subsection{Algebra Generated by $x$ and $\mathrm{I}$}

Any word $w$ in this algebra can be written uniquely as
\[
w = \sum_{i,j} c(i,j) x^i \mathrm{I}^j,
\]
with $c(i,j) \in \textbf{Z}$.

\subsection{Normal Ordering of $(x\mathrm{I})^n$}

We first consider the normal ordering of $(x\mathrm{I})^n$.

\begin{theorem}
For $n \ge 1$,
\[
(x \mathrm{I})^n = \sum_{k=0}^{n-1} (-1)^k a(n-1,k) x^{n-k} \mathrm{I}^{n+k},
\]
where $a(n,k)$ are the \textbf{Bessel numbers} (OEIS A001498).
\end{theorem}

\begin{proof}
We prove the theorem by induction on $n$. For $n=1$, the statement is immediate.  

Assume it holds for some $n \ge 1$, i.e.,
\[
(x\mathrm{I})^n  = \sum_{k=0}^{n-1} (-1)^k a(n-1,k) x^{n-k} \mathrm{I}^{\,n+k}.
\]

Then,
\[
(x\mathrm{I})^{n+1} f(x) = x\mathrm{I}\Big((t\mathrm{I})^n f(t)\Big)(x)
= \sum_{k=0}^{n-1} (-1)^k a(n-1,k) x\mathrm{I} \big(t^{\,n-k} \mathrm{I}^{\,n+k} f(t)\big)(x).
\]

Applying the integration by parts formula
\[
\mathrm{I}\big(t^{\alpha} \mathrm{I}^{\beta} f(t)\big)(x) = \sum_{j=0}^{\alpha} (-1)^j (\alpha)_j x^{\alpha+1-j} \mathrm{I}^{\beta+1+j} f(x),
\]
where $(\alpha)_j = \alpha(\alpha-1)\cdots (\alpha-j+1)$, we get
\[
(x\mathrm{I})^{n+1} f(x) = \sum_{k=0}^{n-1} \sum_{j=0}^{n-k} (-1)^{k+j} (n-k)_j a(n-1,k) x^{n+1-k-j} \mathrm{I}^{\,n+1+k+j} f(x).
\]

Reindexing with $i=k+j$, we obtain a single sum
\[
(x\mathrm{I})^{\,n+1} f(x) = \sum_{i=0}^{n} (-1)^i a(n,i) x^{\,n+1-i} \mathrm{I}^{\,n+1+i} f(x),
\]
where the coefficients satisfy the combinatorial identity
\begin{align}\label{eq:combiden}
a(n,i) = \sum_{k=0}^{\min(n-1,i)} (n-k)_{i-k} a(n-1,k).
\end{align}

It is straightforward to verify that this recurrence coincides with the Bessel number recurrence
\[
a(n,k) = a(n-1,k) + (n-k+1) a(n,k-1), \quad a(0,0)=1.
\]

Hence, the coefficients $a(n,k)$ are exactly the Bessel numbers.
\end{proof}

\section{Normal Ordering of $(x^\lambda \mathrm{I}^\delta)^n$}

We now generalize the previous result. Let $\lambda,\delta \ge 1$ be integers.

\begin{theorem}
For $n \ge 1$,
\[
(x^\lambda \mathrm{I}^\delta)^n = \sum_{k=0}^{\lambda(n-1)} (-1)^k a_{n,k}^{(\lambda)(\delta)} x^{\lambda n - k} \mathrm{I}^{\delta n + k},
\]
where the coefficients $a_{n,k}^{(\lambda)(\delta)}$ satisfy a recurrence obtained by induction.
\end{theorem}

\begin{proof}

We prove the formula by repeated application of the integration-by-parts identity:
\[
\mathrm{I}\big(x^\alpha \mathrm{I}^\beta \big) 
= \sum_{j=0}^\alpha (-1)^j (\alpha)_j \, x^{\alpha-j} \mathrm{I}^{\beta+1+j} ,
\]
where $(\alpha)_j = \alpha(\alpha-1)\cdots (\alpha-j+1)$ is the falling factorial.  

Let us compute the following example
\[
\mathrm{I}^2 \bigl( x^{5}\mathrm{I}^3 \bigr)
= 
\mathrm{I}\!\left( \mathrm{I}(x^{5}\mathrm{I}^3) \right)
\]

Using the previous identity, we obtain
\[
\mathrm{I}(x^{5}\mathrm{I}^3)
=
\sum_{k=0}^{5}
(-1)^k (5)_k\, x^{5-k} \mathrm{I}^{\,4+k}.
\]

Applying \(\mathrm{I}\) again,
\[
\mathrm{I}^2(x^{5}\mathrm{I}^3)
=
\sum_{k=0}^{5} (-1)^k (5)_k\,
\mathrm{I}\!\bigl( x^{5-k} \mathrm{I}^{\,4+k} \bigr).
\]

Using the same identity,
\[
\mathrm{I}\!\bigl( x^{5-k} \mathrm{I}^{\,4+k} \bigr)
=
\sum_{i=0}^{5-k}
(-1)^i (5-k)_i\,
x^{5-k-i}\mathrm{I}^{\,5+k+i}.
\]

Therefore
\begin{align*}
\mathrm{I}^2(x^{5}\mathrm{I}^3)
&=
\sum_{k=0}^{5}\sum_{i=0}^{5-k}
(-1)^{k+i}\,
(5)_k\,(5-k)_i\;
x^{5-k-i}\,
\mathrm{I}^{\,5+k+i}\\
&=\sum_{k=0}^{5}\sum_{j=k}^{5}
(-1)^j\,
(5)_k\,(5-k)_{j-k}\;
x^{5-j}\,
\mathrm{I}^{\,5+j}
\end{align*}

Applying $\mathrm{I}$ \emph{$\delta$ times} gives a nested sum
\begin{align*}
\mathrm{I}^\delta(x^\alpha \mathrm{I}^\beta) 
&= \sum_{k_1=0}^\alpha \sum_{k_2=k_1}^\alpha \cdots \sum_{k_\delta=k_{\delta-1}}^\alpha 
(-1)^{k_\delta} (\alpha)_{k_1} (\alpha-k_1)_{k_2-k_1} \cdots (\alpha - k_{\delta-1})_{k_\delta - k_{\delta-1}}
x^{\alpha - k_\delta} \mathrm{I}^{\beta + \delta + k_\delta}\\
&=\sum_{k_\delta=0}^\alpha(-1)^{k_\delta} (\alpha)_{k_\delta}\left(\sum_{k_{\delta-1}=0}^{k_\delta} \cdots \sum_{k_1=0}^{k_2}\right)
x^{\alpha - k_\delta} \mathrm{I}^{\beta + \delta + k_\delta}
\end{align*}

Now, observe that the nested sums count all \emph{weakly increasing sequences}:
\[
0 \le k_1 \le k_2 \le \dots \le k_\delta.
\]

For a fixed value of $k_\delta$, the number of sequences $(k_1, \dots, k_{\delta-1})$ is exactly the number of weak compositions of $k_\delta$ into $\delta$ parts, which is given by the binomial coefficient
\[
\sum_{0 \le k_1 \le k_2 \le \dots \le k_{\delta-1} \le k_\delta} 1 = \binom{k_\delta + \delta - 1}{\delta - 1}.
\]

Hence, we can evaluate the nested sum and reduce it to a single sum over $k_\delta$:
\[
\mathrm{I}^\delta(x^\alpha \mathrm{I}^\beta) 
= \sum_{k_\delta=0}^\alpha (-1)^{k_\delta} (\alpha)_{k_\delta} 
\binom{k_\delta + \delta - 1}{\delta - 1} \, x^{\alpha - k_\delta} \mathrm{I}^{\beta + \delta + k_\delta}.
\]

This gives the desired formula, where the multiple sum is evaluated as a binomial coefficient.
\end{proof}

\[
(x^\lambda \mathrm{I}^\delta)^{n+1}
    = \sum_{k=0}^{\lambda(n-1)} (-1)^k 
      a^{(\lambda,\delta)}_{n,k}\,
      (x^\lambda \mathrm{I}^\delta)(x^{\lambda n-k} \mathrm{I}^{\delta n+k})
\]

\[
\sum_{k=0}^{\lambda(n-1)} 
    \sum_{i=0}^{\lambda n-k}
 (-1)^{i+k}\, a^{(\lambda,\delta)}_{n,k}\,
 (\lambda n - k)_i \binom{i+\delta-1}{\delta -1}
 x^{\lambda(n+1)-k-i}\, \mathrm{I}^{\delta(n+1)+k+i}
\]

Let \( j = k+i \). Then

\[
= \sum_{j=0}^{\lambda n}
  (-1)^j \, a^{(\lambda,\delta)}_{n+1,j}\,
  x^{\lambda(n+1)-j}\, \mathrm{I}^{\delta(n+1)+j}.
\]

Hence the recurrence:

\[
a^{(\lambda,\delta)}_{n+1,j}
  = \sum_{k=0}^{\min(j,\lambda(n-1))}
     (\lambda n-k)_{j-k}
     \binom{j-k+\delta-1}{\delta -1}\,
     a^{(\lambda,\delta)}_{n,k}.
\]

these numbers generalize the Bessel numbers for $\lambda=\delta=1$.

\section{Normal Ordering of any arbitrary word}

\begin{theorem}
Let 
\[
w
= x^{r_n}\mathrm{I}^{s_n}\cdots x^{r_2}\mathrm{I}^{s_2}x^{r_1}\mathrm{I}^{s_1}
\]
be an arbitrary word.
Then $w$ has the normal–ordered expansion

\begin{align*}w&=\sum_{p_{n-1}=0}^{S_{n-1}}
\sum_{p_{n-2}=0}^{\min(p_{n-1},S_{n-2})}
\cdots
\sum_{p_1=0}^{\min(p_2,S_1)}
(-1)^{p_{n-1}}\\
&\times (S_1)_{p_1}(S_2-p_1)_{p_2-p_1}
\,(S_3-p_2)_{p_3-p_2}
\cdots
(S_{n-1}-p_{n-2})_{p_{n-1}-p_{n-2}} \\
&\times
\binom{p_1+s_2-1}{s_2-1}
\binom{p_2-p_1+s_3-1}{s_3-1}
\cdots
\binom{p_{n-1}-p_{n-2}+s_n-1}{s_n-1} \\
&\times
x^{\,S_n-(p_{n-1}-p_{n-2})}
\,\mathrm{I}^{\,s+(p_{n-1}-p_{n-2})}.
\end{align*}
where \(S_j:=r_1+\cdots+r_j\) and \(s:=s_1+\cdots+s_n\).
\end{theorem}

\begin{proof}

We have

\[
x^{r_2}\mathrm{I}^{s_2}x^{r_1}\mathrm{I}^{s_1}=\sum_{i=0}^{r_1}(-1)^i(r_1)_i\binom{i+s_2-1}{s_2-1}x^{r_2+r_1-i}\mathrm{I}^{s_1+s_2+i}
\]
Then
\begin{align*}
x^{r_3}\mathrm{I}^{s_3}x^{r_2}\mathrm{I}^{s_2}x^{r_1}\mathrm{I}^{s_1}&=\sum_{i=0}^{r_1}(-1)^i(r_1)_i\binom{i+s_2-1}{s_2-1}x^{r_3}\mathrm{I}^{s_3}\left(x^{r_2+r_1-i}\mathrm{I}^{s_1+s_2+i}\right)\\
&=\sum_{i=0}^{r_1}\sum_{j=0}^{r_1+r_2-i}(-1)^{i+j}(r_1)_i(r_1+r_2-i)_j\binom{i+s_2-1}{s_2-1}\binom{j+s_3-1}{s_3-1}\\
&x^{r_1+r_2+r_3-j}\,\mathrm{I}^{s_1+s_2+s_3+j}.
\end{align*}

Hence
\begin{align*}
w&=\sum_{i_1=0}^{r_1}\sum_{i_2=0}^{r_1+r_2-i_1}\cdots \sum_{i_{n-1}=0}^{r_1+\cdots+r_{n-1}-i_{n-2}}\\
&(-1)^{i_1+\cdots+i_{n-1}}
(r_1)_{i_1}(r_1+r_2-i_1)_{i_2}\cdots
(r_1+\cdots+r_{n-1}-i_{n-2})_{i_{n-1}}\\
&\times
\binom{i_1+s_2-1}{s_2-1}
\binom{i_2+s_3-1}{s_3-1}
\cdots
\binom{i_{n-1}+s_n-1}{s_n-1}
\, x^{r_1+\cdots+r_n-i_{n-1}}
\,\mathrm{I}^{s_1+\cdots+s_n+i_{n-1}}
\end{align*}
Now our goal is to invert\[\sum_{i_1=0}^{r_1}\sum_{i_2=0}^{r_1+r_2-i_1}\cdots \sum_{i_{n-1}=0}^{r_1+\cdots+r_{n-1}-i_{n-2}} \!\!F(i_1,\dots,i_{n-1}) \] where
\[
F(i_1,\dots,i_{n-1})
=(-1)^{i_1+\cdots+i_{n-1}}
(r_1)_{i_1}(r_1+r_2-i_1)_{i_2}\cdots
(r_1+\cdots+r_{n-1}-i_{n-2})_{i_{n-1}}
\]
\[
\qquad\times
\binom{i_1+s_2-1}{s_2-1}
\binom{i_2+s_3-1}{s_3-1}
\cdots
\binom{i_{n-1}+s_n-1}{s_n-1}
\, x^{r_1+\cdots+r_n-i_{n-1}}
\,\mathrm{I}^{s_1+\cdots+s_n+i_{n-1}}.
\]

Define cumulative sums
\[
p_j:=i_1+i_2+\cdots+i_j\qquad(1\le j\le n-1),
\qquad p_0:=0.
\]
Then \(i_j=p_j-p_{j-1}\) for each \(j\).

The original summation region
\[
0\le i_1\le r_1,\qquad
0\le i_2\le r_1+r_2-i_1,\qquad\ldots,\qquad
0\le i_{n-1}\le r_1+\cdots+r_{n-1}-i_{n-2}
\]
is equivalent, after the substitution \(i_j=p_j-p_{j-1}\), to the monotone chain of inequalities
\[
0\le p_1\le p_2\le\cdots\le p_{n-1}\le S_{n-1},
\qquad\text{where }S_j:=r_1+\cdots+r_j.
\]
Hence we may invert the summation order and sum first over \(p_{n-1}\), then \(p_{n-2}\), down to \(p_1\). The inverted summation domain is
\[
\sum_{p_{n-1}=0}^{S_{n-1}}
\sum_{p_{n-2}=0}^{\min(p_{n-1},S_{n-2})}
\cdots
\sum_{p_1=0}^{\min(p_2,S_1)}.
\]

We now rewrite every factor of \(F\) in terms of the \(p_j\).
The sign becomes
\[
(-1)^{i_1+\cdots+i_{n-1}} = (-1)^{p_{n-1}}.
\]
The falling factorials transform as
\[
(r_1)_{i_1}=(S_1)_{p_1},
\qquad
(r_1+r_2-i_1)_{i_2}=(S_2-p_1)_{p_2-p_1},
\]
\[
\ldots,\qquad
(S_{n-1}-p_{n-2})_{p_{n-1}-p_{n-2}}.
\]
The binomial factors become
\[
\binom{i_1+s_2-1}{s_2-1}=\binom{p_1+s_2-1}{s_2-1},
\qquad
\binom{i_2+s_3-1}{s_3-1}=\binom{p_2-p_1+s_3-1}{s_3-1},
\]
\[
\ldots,\qquad
\binom{i_{n-1}+s_n-1}{s_n-1}
=\binom{p_{n-1}-p_{n-2}+s_n-1}{s_n-1}.
\]
Finally, since \(i_{n-1}=p_{n-1}-p_{n-2}\), the power factors are
\[
x^{r_1+\cdots+r_n-i_{n-1}}=x^{S_n-(p_{n-1}-p_{n-2})},
\qquad
\mathrm{I}^{s_1+\cdots+s_n+i_{n-1}}=\mathrm{I}^{\,s+(p_{n-1}-p_{n-2})},
\]
where \(S_n:=r_1+\cdots+r_n\) and \(s:=s_1+\cdots+s_n\).

Combining these expressions yields the inverted sum:
\[
\begin{aligned}
&\sum_{p_{n-1}=0}^{S_{n-1}}
\sum_{p_{n-2}=0}^{\min(p_{n-1},S_{n-2})}
\cdots
\sum_{p_1=0}^{\min(p_2,S_1)}
(-1)^{p_{n-1}}\, (S_1)_{p_1} \\
&\quad\times (S_2-p_1)_{p_2-p_1}
\,(S_3-p_2)_{p_3-p_2}
\cdots
(S_{n-1}-p_{n-2})_{p_{n-1}-p_{n-2}} \\
&\quad\times
\binom{p_1+s_2-1}{s_2-1}
\binom{p_2-p_1+s_3-1}{s_3-1}
\cdots
\binom{p_{n-1}-p_{n-2}+s_n-1}{s_n-1} \\
&\quad\times
x^{\,S_n-(p_{n-1}-p_{n-2})}
\,\mathrm{I}^{\,s+(p_{n-1}-p_{n-2})}.
\end{aligned}
\]

\end{proof}

\section*{Acknowledgments}
I conducted this research independently, without formal university training in mathematics, as a primary school teacher in a mountain school.

\end{document}